%% file: main.tex
\documentclass{article}
\usepackage[a4paper, margin=3cm, bottom=4cm]{geometry}
\usepackage[utf8]{inputenc}
\usepackage[T1]{fontenc}
\usepackage{lmodern}

\usepackage[
    backend=biber,
    style=alphabetic, maxalphanames=5, % write all initials, e.g "FBHW21" instead of "Foo+21"
    maxbibnames=100, maxsortnames=100, % disable "et al." in \printbibliography
    sorting=nyt,
    natbib=true,
    url=false,doi=false,isbn=false,eprint=false
]{biblatex}
\usepackage[pdftex]{graphicx} % "[pdftex]" for arxiv compiler
\usepackage{subcaption}
\graphicspath{ {./images/} }
\usepackage{enumitem}
\usepackage[normalem]{ulem}
\usepackage[toc,page]{appendix}

\usepackage{tikz}
\usetikzlibrary{
    arrows.meta,
    backgrounds,
    calc,
    fit
}

\usepackage[bookmarks, bookmarksnumbered,
			colorlinks=true,
            linkcolor = blue,
            urlcolor  = blue,
            citecolor = teal,
            hypertexnames=false  % workaround for texlive 2025's bad handling of autoref with numberwithin
]{hyperref}

\newcommand\fnsurl[1]{{\footnotesize\url{#1}}}

\usepackage{myquicksetup}
\numberwithin{definition}{section}
\numberwithin{theorem}{section}
\numberwithin{corollary}{section}
\numberwithin{proposition}{section}
\numberwithin{lemma}{section}
\numberwithin{claim}{section}
\numberwithin{fact}{section}
\numberwithin{remark}{section}
\numberwithin{example}{section}
\numberwithin{equation}{section}
\mathtoolsset{showonlyrefs} % only display equation numbers for equations that are referenced to. https://tex.stackexchange.com/a/4729

\usepackage{mylivemacros}
\usepackage{sinkhorn_small} % project-specific macros

\newif\ifextended  % extended version: include all appendices
\extendedtrue
\title{\vspace{-1em} Sharp $O(1/k)$ convergence rate for the Sinkhorn algorithm\\ via a local analysis}
\date{\today}

\author{Guillaume Wang\footnote{Courant Institute School, New York
University~ \texttt{guillaume.wang@nyu.edu}}}

\hypersetup{
    pdftitle={Sharp O(1/k) convergence rate for the Sinkhorn algorithm
    via a local analysis},
    pdfauthor={Guillaume Wang},
    pdfkeywords={Sinkhorn algorithm, matrix scaling, entropic optimal transport}
}

\begin{document}
\maketitle

% \subfile{sections/abstract}
% \subfile{sections/1_intro}
% \subfile{sections/2_}
% \subfile{sections/3_}
% \subfile{sections/acknowledgments}

\begin{abstract}
    We prove that the Sinkhorn algorithm converges at the rate of $O(1/k)$ in $\ell_1$-norm marginal error and in joint relative entropy, which is known to be sharp in the asymptotically scalable case.
    The proof is based on examining the bipartite graph associated to the entropy-regularized optimal transport problem, and treating differently the edges that are assigned a positive mass in the optimal transport plan vs.\ those that are not.
    This yields a local convergence bound with the sharp rate, which is bootstrapped into a global bound using the author's previous result in \cite{wang2026almost} where we showed an almost-sharp rate up to a logarithmic factor.
    % The same rate is also shown for the sequence of even iterates (resp.\ odd iterates) in the non-scalable case, i.e., when there is no feasible transport plan.
\end{abstract}

\section{Introduction} \label{sec:intro}

Let $\mu \in \Delta_m$ and $\nu \in \Delta_n$, where $\Delta_m$ denotes the probability simplex in dimension $m$, such that $\mu_{\min} = \min_i \mu_i, \nu_{\min} = \min_j \nu_j > 0$.
Let $C \in (\RR \cup \{+\infty\})^{m \times n}$ and $\EEE = \left\{ (i,j);~ C_{ij} < \infty \right\}$, and suppose that the bipartite graph $(\{1 \dots m \} \sqcup \{1 \dots n\}, \EEE)$ 
% is connected.
has no isolated vertex.
Let $\tau>0$ and consider
the entropy-regularized optimal transport (EOT) problem
\begin{equation} \label{eq:intro:EOT}
    \min_{\pi \in \Delta_{\EEE}}~
    \sum_{(i,j) \in \EEE} C_{ij} \pi_{ij} + \tau \Hdiv{\pi}{\mu \otimes \nu}
    ~~~~\text{subject to}~~~~
    X_\sharp \pi = \mu
    ~~\text{and}~~
    Y_\sharp \pi = \nu,
\end{equation}
where
for all $\pi \in \Delta_\EEE$,
\begin{equation}
    X_\sharp \pi = \bigg( \sum_{j: (i,j) \in \EEE} \pi_{ij} \bigg)_i \in \Delta_m
    ~~\qquad\text{and}\qquad~~
    Y_\sharp \pi = \bigg( \sum_{i: (i,j) \in \EEE} \pi_{ij} \bigg)_j \in \Delta_n
\end{equation}
denote the first and second marginals, and 
$\Hdiv{\pi}{\mu \otimes \nu} = \sum_{(i,j) \in \EEE} \pi_{ij} \log \frac{\pi_{ij}}{\mu_i \nu_j}$
denotes the relative entropy w.r.t.\ a product distribution.

% Throughout this paper except for \autoref{sec:nonscal},
Throughout this paper,
we assume that \eqref{eq:intro:EOT} admits a feasible solution.%
\footnote{To see that the existence of a feasible solution is not guaranteed, as a simple example, consider the case where $m=n=2$, $\EEE = \{ (1,1), (1,2), (2,2) \}$, and $\mu_2 > \nu_2$. Then one cannot have both $\pi_{22} = \mu_2$ and $\pi_{12} + \pi_{22} = \nu_2$.}
By strict convexity of $H$, there exists a unique optimal solution, denoted $\pi^*$, and we denote its support by $\SSS = \left\{ (i,j);~ \pi^*_{ij} > 0 \right\}$.
Note that by definition $\SSS \subset \EEE$, and the inclusion can be strict.
In other words, following the terminology of matrix scaling, we assume we are in the \emph{asymptotically scalable case} \cite[Section~4]{idel2016review}.

The primal iterates of the Sinkhorn algorithm---also known as RAS method or iterative proportional fitting---are the $\pi^k \in \Delta_{\EEE}$, for $k \geq 0$, defined by
$\pi^0_{ij} \propto e^{-C_{ij}/\tau} \mu_i \nu_j$ and
\begin{align} \label{eq:intro:primal_upd}
    &\text{for $k$ even,}~~
    \pi^{k+1}_{ij}
    = \frac{\mu_i}{(X_\sharp \pi^k)_i}\, \pi^k_{ij}
    &
    &\text{and}
    &
    &\text{for $k$ odd,}~~
    \pi^{k+1}_{ij} = \frac{\nu_j}{(Y_\sharp \pi^k)_j}\, \pi^k_{ij}
\end{align}
for all $(i,j) \in \EEE$.
The optimality metric we use to measure the eventual convergence of the Sinkhorn algorithm is the $\ell_1$-norm error of the marginals,
as is standard in the literature:
\begin{equation}
    E_k = \norm{X_\sharp \pi^k - \mu}_1 + \norm{Y_\sharp \pi^k - \nu}_1.
\end{equation}
Indeed one can show that due to the specific choice of $\pi^0$, $\pi^k \to \pi^*$ if and only if $E_k \to 0$ \cite{altschuler2017near}.

\paragraph{State of the art and contributions.}
It is known since the work of \citet{sinkhorn1967concerning} that $\pi^k \to \pi^*$ as $k \to \infty$, that is, $E_k = o(1)$.
\citet[page~19]{soules1991rate} gave an explicit example of a problem instance for which $E_k = \Theta(1/k)$.
\citet{leger2021gradient} showed that $E_k \leq O(1/\sqrt{k})$.
\citet[Proposition~3]{qu2025sinkhorn} showed that if $\SSS \neq \EEE$, then necessarily $E_k \geq \Omega(1/k)$.
% (see also \cite{achilles1993implications} for an earlier related negative result).
The author showed recently in \cite{wang2026almost} that $E_k \leq O((\log k)/k)$.
We refer the reader to that last reference for a more detailed description of the state of the art in the cases where $\SSS = \EEE$ or $\SSS = \EEE = \{1 \dots m \} \times \{1 \dots n\}$,
since our focus here is on the general case where both inclusions can be strict.

In this paper, we close the logarithmic gap between the upper bound of \cite{wang2026almost} and the lower bound of \cite{qu2025sinkhorn}, by showing that $E_k \leq O(1/k)$. Formally, we show the following.
A more precise version with explicit constants is given in \autoref{thm:mainres:glob}.

\begin{theorem} \label{thm:intro:mainres}
    Suppose that \eqref{eq:intro:EOT} admits a feasible solution.
    There exist constants $B_1, B_2>0$ dependent only on $\mu, \nu$, and $\EEE$ such that
    \begin{equation}
        \forall k \geq B_1,~~
        E_k \leq \frac{B_2}{k} \left( 1 + \frac{\max_{(i,j) \in \EEE} C_{ij} - \min_{ij} C_{ij}}{\tau} \right).
    \end{equation}
\end{theorem}

Although our result subsumes that of \cite{wang2026almost}, let us emphasize that our result builds on top of theirs.
Indeed, our analysis in this paper is local in nature:
our main technical result is an inequality of the form
$E_k \leq \frac{A}{k}$ valid only for all $k$ such that $\Hdiv{\pi^*}{\pi^k} \leq B$, for certain constants $A$ and $B$.
To deduce a computational complexity bound, this must still be combined with an upper estimate on $k_0$, the first iteration such that $\Hdiv{\pi^*}{\pi^k} \leq B$.
Such an estimate is obtained precisely thanks to the result of \cite{wang2026almost}.%
% \footnote{A bound on $k_0$ could also be obtained using the earlier work of \cite{leger2021gradient}, but this would result in a worse dependency on the problem parameters.}  ---> not even, actually
\footnote{We also stress that a bound on $k_0$ could not have been obtained easily using earlier works: prior to \cite{wang2026almost}, convergence bounds for the Sinkhorn algorithm in the asymptotically scalable case were established only for the marginal errors $E_k$ \cite{leger2021gradient}, whereas here we needed a bound on the \emph{joint} relative entropy.
% (The bound on joint relative entropy shown in \cite[Proposition~7]{aubin2022mirror} is only applicable in the finite-cost case.)
}

% We also prove an extension of the above theorem to the case where \eqref{eq:intro:EOT} has no feasible solution.
% Its statement is delayed to \autoref{sec:nonscal}, as it requires introducing a certain amount of background and notations.
% In a nutshell, it is known that in this case, $\pi^{2k}$ and $\pi^{2k+1}$ converge respectively to some $\pi^*_{\mathrm{even}}$ and $\pi^*_{\mathrm{odd}}$ as $k \to \infty$ \cite{aas2014limit,baradat2024convergence}, and we prove that the convergence occurs at the sharp rate of $O(1/k)$.

\vspace{1em}
The paper is organized as follows.
In \autoref{sec:prelim}, we recall a number of auxiliary facts concerning the structure of the EOT problem and the Sinkhorn algorithm.
In \autoref{sec:mainres}, we state and prove (a more precise version of) our main result.
% In \autoref{sec:nonscal}, we state and prove an extension to the case where \eqref{eq:intro:EOT} has no feasible solution.
We conclude in \autoref{sec:ccl} with discussions and possible directions for future work.

\section{Preliminaries} \label{sec:prelim}

Throughout this paper, we are interested in the optimization problem defined in \eqref{eq:intro:EOT}.
Recall that we assume it admits a feasible solution and hence a unique optimal solution $\pi^*$, whose support is denoted by $\SSS = \left\{ (i,j);~ \pi^*_{ij} > 0 \right\} \subset \EEE$.

\subsection{Dual Sinkhorn iterates} \label{subsec:prelim:dual}

The convex dual of the EOT problem \eqref{eq:intro:EOT} is
% $\max_{f,g}~ (-\Phi(f,g))$ where
% \begin{equation}
%     \forall (f,g) \in \RR^m \times \RR^n,~~
%     \Phi(f, g) =
%     \tau \log \sum_{ij} e^{\left[ -C_{ij} + f_i + g_j \right]/\tau} \mu_i \nu_j
%     - \mu^\top f - \nu^\top g
% \end{equation}
% with the convention that $\exp(-\infty) = 0$, i.e., the sum can equivalently be taken over the $(i,j) \in \EEE$.
% Also consider the alternative dual formulation, which is obtained by ignoring the normalizing constraint $\sum_{ij} \pi_{ij}=1$ when dualizing:
$\max_{f,g}~ (-\Psi(f,g))$ where
\begin{equation}
    \forall (f,g) \in \RR^m \times \RR^n,~~
    \Psi(f, g) =
    \tau \sum_{ij} e^{\left[ -C_{ij} + f_i + g_j \right]/\tau} \mu_i \nu_j - \tau - \mu^\top f - \nu^\top g,
\end{equation}
with the convention that $\exp(-\infty) = 0$, i.e., the sum can equivalently be taken over the $(i,j) \in \EEE$.
The dual Sinkhorn iterates are the $(f^k, g^k) \in \RR^m \times \RR^n$ for $k \geq 0$ defined by $(f^0, g^0) = (0,0)$ and
\begin{align}
    &\text{for $k$ even,}~~~
    f^{k+1} = f[g^k]
    \quad \text{and} \quad
    g^{k+1} = g^k \\
    &\text{for $k$ odd,}~~~~
    f^{k+1} = f^k
    \quad\quad \text{and} \quad
    g^{k+1} = g[f^k]
\end{align}
where
\begin{equation}
    \forall g \in \RR^n,~
    f[g]_i = -\tau \log \sum_j\, e^{\left[ -C_{ij} + g_j \right]/\tau} \nu_j
    \qquad \text{and} \qquad
    \forall f \in \RR^m,~
    g[f]_j = -\tau \log \sum_i\, e^{\left[ -C_{ij} + f_i \right]/\tau} \mu_i,
\end{equation}
still with the convention that $\exp(-\infty) = 0$.
Note that 
$\forall g, f[g] = \argmin \Psi(\cdot, g)$
and
$\forall f, g[f] = \argmin \Psi(f, \cdot)$.
% $\forall g, \argmin \Phi(\cdot, g) = \{ f[g] + b \bmone_m, b \in \RR \}$
% and
% $\forall f, \argmin \Phi(f, \cdot) = \{ g[f] + b \bmone_n, b \in \RR \}$.

\pagebreak

For any $(f, g) \in \RR^m \times \RR^n$, denote
\begin{equation}
    \pi[f,g]
    = \left(
        \frac{1}{Z(f,g)} e^{\left[ -C_{ij} + f_i + g_j \right]/\tau} \mu_i \nu_j
    \right)_{ij}
    \in \Delta_\EEE
    \quad\text{where}\quad
    Z(f,g) = \sum_{i'j'} e^{\left[ -C_{i'j'} + f_{i'} + g_{j'} \right]/\tau} \mu_{i'} \nu_{j'}.
\end{equation}
Then one can check that $\pi[f^k, g^k] = \pi^k$ for all $k \geq 0$.
Moreover, one has $Z(f^k, g^k) = 1$ for all $k \geq 1$, so that $\pi^k_{ij} = e^{[-C_{ij} + f^k_i + g^k_j]/\tau} \mu_i \nu_j$ for all $k \geq 1$.

% A number of previous works analyzing the convergence of the Sinkhorn algorithm, including \cite{wang2026almost}, are based primarily on this dual point of view. While its allows for an easier translation of classical intuitions from optimization---e.g., from the study of \cite{beck2013convergence} on alternating minimization---this point of view proves slightly awkward in the case where $\SSS \neq \EEE$.
While this dual point of view allows for an easier translation of classical intuitions from optimization, it proves slightly awkward in the case where $\SSS \neq \EEE$.
Indeed, in this case, $\Psi$ does not attain its infimum, because $\pi^*$ cannot be of the form $\pi^* = \pi[f,g]$ for any finite vectors $f,g$.
The converse is also true: if $\SSS = \EEE$ then there exists $(f^*, g^*)$ such that $\pi^* = \pi[f^*,g^*]$ and $\Psi(f^*, g^*) = \inf \Psi$.

\subsection{Optimality metrics}

As mentioned in the introduction, the primary optimality metric we use to measure the eventual convergence of the Sinkhorn algorithm is $E_k$, the $\ell_1$-norm error of the marginals.
The following alternative metrics will also play an important role in our analysis:
\begin{align}
    G_k^2 &= \Hdiv{\mu}{X_\sharp \pi^k} + \Hdiv{\nu}{Y_\sharp \pi^k},
    &
    V_k &= \tau \Hdiv{\pi^*}{\pi^k},
\end{align}
where
$\Hdiv{\pi}{\pi'} = \sum_I \pi_I \log \frac{\pi_I}{\pi'_I}$
is the relative entropy between any discrete probability distributions.
% \TODO{actually do we really need $G_k^2$?}  ---> not reeaaally... but still easier to leave it, both for me and for the reader IMO.

Let us recall some basic facts about the relation between these three metrics and the dual objective $\Psi$.

\begin{lemma} \label{lm:prelim:Ek_Gk}
    For any $k \geq 1$, if $k$ is odd then $X_\sharp \pi^k = \mu$, and if $k$ is even then $Y_\sharp \pi^k = \nu$.
    In particular,
    $E_k = \norm{X_\sharp \pi^k - \mu}_1 \vee \norm{Y_\sharp \pi^k - \nu}_1$
    and
    $G_k^2 = \Hdiv{\mu}{X_\sharp \pi^k} \vee \Hdiv{\nu}{Y_\sharp \pi^k}$,
    and moreover
    \begin{equation}
        \frac12 E_k^2 \leq G_k^2.
    \end{equation}
\end{lemma}

\begin{proof}
    The first part of the lemma follows directly from the definition of the primal Sinkhorn update \eqref{eq:intro:primal_upd}.
    The second part then follows from Pinsker's inequality: for any $k \geq 1$,
    $\frac12 E_k^2 = \frac12 \left[ \norm{X_\sharp \pi^k - \mu}_1^2 \vee \norm{Y_\sharp \pi^k - \nu}_1^2 \right]
    \leq \Hdiv{\mu}{X_\sharp \pi^k} \vee \Hdiv{\nu}{Y_\sharp \pi^k} = G_k^2$.
\end{proof}

\begin{lemma} \label{lm:prelim:Vk_Psi}
    For any $(f,g) \in \RR^m \times \RR^n$ such that $Z(f,g) = 1$,
    $\tau \Hdiv{\pi^*}{\pi[f,g]} = \Psi(f,g) - \inf \Psi$.
    In particular for any $k \geq 1$,
    $V_k = \Psi(f^k, g^k) - \inf \Psi$.
\end{lemma}

% \begin{proof}[Proof sketch]
%     Here we only show the lemma under the additional assumption that $\SSS = \EEE$, i.e., that $\Psi$ attains its infimum and $\pi^*$ is of the form $\pi[f^*, g^*]$ for some finite vectors $f^* \in \RR^m, g^* \in \RR^n$. The proof for the general case proceeds by an approximation argument, delayed to \autoref{sec:apx_delayedpfs}.
% 
%     For any $f,g$ such that $Z(f,g) = 1$, by definition of $\pi[f,g]_{ij} = e^{[-C_{ij} + f_i + g_j]/\tau} \mu_i \nu_j$,
%     \begin{multline*}
%         \tau \Hdiv{\pi^*}{\pi[f,g]}
%         = \sum_{ij} \pi^*_{ij} \left( f^*_i+g^*_j - f_i-g_j \right)
%         = (X_\sharp \pi^*)^\top (f^*-f)
%         + (Y_\sharp \pi^*)^\top (g^*-g) \\
%         = \mu^\top (f^* - f) + \nu^\top (g^* - g)
%         = \Psi(f, g) - \Psi(f^*, g^*),
%     \end{multline*}
%     as announced.
% \end{proof}

\begin{proof}
    Fix any $f,g$ such that $Z(f,g) = 1$ and let for concision $\pi = \pi[f,g] = \left( e^{[-C_{ij} + f_i + g_j]/\tau} \mu_i \nu_j \right)_{ij}$.
    Then $\Psi(f,g) = 0 - f^\top \mu - g^\top \nu$ and
    \begin{align*}
        \forall (i,j) \in \EEE,~~~
        \log \frac{\pi^*_{ij}}{\pi_{ij}}
        &= \log \frac{\pi^*_{ij}}{\mu_i \nu_j}
        - \log \frac{\pi_{ij}}{\mu_i \nu_j}
        = \log \frac{\pi^*_{ij}}{\mu_i \nu_j}
        + \frac{C_{ij} - f_i - g_j}{\tau} \\
        \tau \Hdiv{\pi^*}{\pi[f,g]}
        &= \tau \sum_{ij} \pi^*_{ij} 
        \log \frac{\pi^*_{ij}}{\pi_{ij}}
        = ~\underbrace{
            \tau \Hdiv{\pi^*}{\mu \otimes \nu}
            + \sum_{ij} C_{ij} \pi^*_{ij}
        }_{=\, \min \eqref{eq:intro:EOT} \,=\, -\inf \Psi}~
        -\, f^\top \underbrace{(X_\sharp \pi^*)}_{\mu}~
        -\, g^\top \underbrace{(Y_\sharp \pi^*)}_{\nu} \\
        &= -\inf \Psi + \Psi(f,g)
    \end{align*}
    by optimality of $\pi^*$ and by definition of $\max_{f,g}~ (-\Psi(f,g))$ as the dual problem of \eqref{eq:intro:EOT}.
\end{proof}

\begin{lemma}[{\cite[Lemma~2]{altschuler2017near}}] \label{lm:prelim:Vk+1_Vk_Gk2}
    For any $k \geq 1$,
    $V_{k+1} - V_k = -\tau G_k^2$.
\end{lemma}

\begin{remark}
    If $(x^k)_k$ are the iterates of gradient descent on an objective function $f: \RR^d \to \RR$, i.e., $x^{k+1} = x^k - \eta \nabla f(x^k)$ for a step-size $\eta>0$, then we have
    $f(x^{k+1}) - f(x^k) = -\eta \norm{f(x^k)}^2 + O(\eta^2)$.
    So in view of \autoref{lm:prelim:Vk_Psi} and \autoref{lm:prelim:Vk+1_Vk_Gk2},
    $G_k^2$ acts as an analog of the squared gradient norm for gradient descent.
    % , hence the choice of notation.
\end{remark}

% \begin{lemma}[{\cite[Proposition~6.10]{nutz2021introduction}}]
%     The sequences $(G_{2k})_k$ and $(G_{2k+1})_k$ are non-increasing.
%     \TODO{remove if not used}
% \end{lemma}

\begin{lemma}[``doubling trick''] \label{lm:prelim:Ek2_Vk/2}
    For any $k \geq 2$,
    $E_k^2 \leq \frac{8}{\tau k} V_{\ceil{k/2}}$.
\end{lemma}

\begin{proof}
    See \cite[Proposition~4.3]{ghosal2025convergence} or \cite[Lemma~2.5]{wang2026almost}.
\end{proof}

Finally, let us remark the following useful decomposition of the joint forward relative entropy $V_k/\tau = \Hdiv{\pi^*}{\pi^k}$.

\begin{lemma} \label{lm:prelim:decomp_Vk}
    For any $\pi \in \Delta_\EEE$,
    \begin{equation}
        \Hdiv{\pi^*}{\pi}
        = \sum_{(i,j) \in \SSS} 
        D_h(\pi^*_{ij} \| \pi_{ij})
        + \sum_{(i,j) \in \EEE\setminus\SSS} \pi_{ij}
    \end{equation}
    where for all $p, q>0$,
    $D_h(p \| q) = p \log \frac{p}{q} - p + q$.
    Note that $D_h(p \| q) \geq 0$ with equality if and only if $p=q$.
\end{lemma}

\begin{proof}
    The lemma follows from direct computations.
\end{proof}

\subsection{Asymptotic and exact scalability}

In the matrix scaling literature, a matrix $A \in \RR_+^{m \times n}$ is called exactly $(\mu,\nu)$-scalable if there exist $(f^*, g^*) \in \RR^{m \times n}$ such that $A' = \big( A_{ij} e^{f^*_i + g^*_j} \big)_{ij}$ satisfies $A' \bmone_n = \mu$ and $A^{\prime \top} \bmone_m = \nu$.
% (where $\bmone_N$ is the vector in $\RR^N$ with all components equal to $1$).
Similarly, $A$ is called asymptotically $(\mu, \nu)$-scalable if there exists a sequence $(f^l, g^l)_l$ such that $A^l = \big( A_{ij} e^{f^l_i + g^l_j} \big)_{ij}$ converges and its limit $A'$ satisfies $A' \bmone_n = \mu$ and $A^{\prime \top} \bmone_m = \nu$; note that $A'$ may have a smaller support set than $A$.
As one could expect given the dual perspective of \autoref{subsec:prelim:dual}, 
solving \eqref{eq:intro:EOT} is precisely equivalent to finding such a matrix $A'$ for $A = \left( e^{-C_{ij}/\tau} \mu_i \nu_j \right)_{ij}$ \cite{idel2016review}.

In fact, it is known since at least the 1960's that exact/asymptotic $(\mu, \nu)$-scalability is a property of the support set or ``pattern'' of $A$, independent of its specific non-zero coefficients \cite[Theorems~4.1, 4.2]{idel2016review}.
This justifies the following definition, which is a bit more natural for our purposes.
Although not essential for this paper's derivations, we find that it helps to clarify ideas and to prevent possible confusion.
% , especially for \autoref{sec:nonscal}.

\begin{definition} \label{def:prelim:scalable_triplet}
    Consider any $\mu \in \Delta_m, \nu \in \Delta_n$ such that $\mu_{\min}, \nu_{\min}>0$
    and any $\EEE \subset \{1 \dots m\} \times \{1 \dots n\}$ such that the bipartite graph with edge set $\EEE$ has no isolated vertex.
    We say that
    \begin{itemize}
        \item $(\mu, \nu, \EEE)$ is \emph{asymptotically scalable} if there exists $\pi \in \Delta_\EEE$ such that $X_\sharp \pi = \mu$ and $Y_\sharp \pi = \nu$.
        In this case, we define the \emph{optimal support set} as the subset
        \begin{equation}
            \SSS(\mu, \nu, \EEE) 
            = \left\{
                (i,j) \in \EEE;~~~
                \exists \pi \in \Delta_\EEE 
                ~~\text{s.t.}~~
                X_\sharp \pi = \mu,~ 
                Y_\sharp \pi = \nu,~
                \pi_{ij} > 0
            \right\}.
        \end{equation}
        \item $(\mu, \nu, \EEE)$ is \emph{exactly scalable} if it is asymptotically scalable and
        $\SSS(\mu, \nu, \EEE) = \EEE$.
    \end{itemize}
\end{definition}

% In other words, we say that $(\mu, \nu, \EEE)$ is [asymptotically/exactly] scalable if $A$ is [asymptotically/exactly] $(\mu, \nu)$-scalable for all---equivalently, for some---$A$ with support set $\EEE$.

Coming back to the problem setup considered in this paper as introduced in \autoref{sec:intro}, our assumption that the EOT problem \eqref{eq:intro:EOT} admits feasible solutions is equivalent to assuming $(\mu, \nu, \EEE)$ asymptotically scalable. 
Moreover the set $\SSS = \left\{ (i,j) \in \EEE;~ \pi^*_{ij}>0 \right\}$ is precisely the optimal support set of $(\mu, \nu, \EEE)$, as the following lemma formalizes.

\begin{lemma} \label{lm:prelim:piopt_SSS}
    Consider an asymptotically scalable triplet $(\mu, \nu, \EEE)$.
    For any finite coefficients $(C_{ij})_{ij} \subset \RR^\EEE$ and $\tau>0$,
    the optimal solution $\pi^*$ of \eqref{eq:intro:EOT} satisfies
    $\left\{ (i,j) \in \EEE;~ \pi^*_{ij}>0 \right\} = \SSS(\mu, \nu, \EEE)$.
    
    In particular, for any asymptotically scalable triplet $(\mu, \nu, \EEE)$, there exists $\pi \in \Delta_{\EEE}$ such that $X_\sharp \pi = \mu$, $Y_\sharp \pi = \nu$, and $\forall (i,j) \in \SSS(\mu, \nu, \EEE),~ \pi_{ij}>0$.
\end{lemma}

\begin{proof}
    The inclusion $\left\{ (i,j) \in \EEE;~ \pi^*_{ij}>0 \right\} \subset \SSS(\mu, \nu, \EEE)$ holds by definition of $\SSS(\mu, \nu, \EEE)$.
    
    Conversely, suppose by contradiction that there exists $(i_0,j_0) \in \SSS(\mu, \nu, \EEE)$ such that $\pi^*_{i_0j_0} = 0$.
    By definition, there exists $\pi' \in \Delta_\EEE$ such that $X_\sharp \pi' = \mu$, $Y_\sharp \pi' = \nu$ and $\pi'_{i_0j_0}>0$.
    Denote $\pi_t = (1-t) \pi^* + t \pi'$ and remark that $\pi_t$ is feasible for \eqref{eq:intro:EOT} for all $0 \leq t \leq 1$.
    Then one can check that $h: t \mapsto \Hdiv{\pi_t}{\mu \otimes \nu}$
    satisfies $\lim_{t \to 0^+} h'(t) = -\infty$. So there exists $t \neq 0$ such that $\pi_t$ achieves a smaller objective value for \eqref{eq:intro:EOT} than $\pi^*$, contradicting the definition of $\pi^*$.

    The second part of the lemma follows immediately, as it is suffices to choose as $\pi$ the optimal solution of the EOT problem with marginals $\mu, \nu$ and any matrix $C$ with finite-entry pattern $\EEE$.
\end{proof}

\section{Main result} \label{sec:mainres}

In this section, we state and prove a more precise version of \autoref{thm:intro:mainres}.
To do so, we start by describing in detail the structure of the optimal support set $\SSS$ in \autoref{subsec:mainres:struct}.
Then we prove some key lemmas on the local convergence behavior of the Sinkhorn algorithm in \autoref{subsec:mainres:loc}.
Finally, we present the statement and proof of our main result in \autoref{subsec:mainres:glob}.

\subsection{Structure of the admissible and optimal support sets} \label{subsec:mainres:struct}

The following decomposition of the admissible support set $\EEE$ is due to \cite{allen2017much}.
Its connection to the Dulmage-Mendelsohn decomposition of bipartite graphs was remarked by \cite{hayashi2024finding}.
We use the same choice of notations and terminology as in \cite{wang2026almost}.
See \autoref{fig:mainres:DM_E} for an illustration.

\begin{proposition}[{\cite[Lemma~C.3]{allen2017much}}] \label{prop:mainres:DM_E}
    Suppose that $(\mu, \nu, \EEE)$ is asymptotically scalable.
    Then there exist an integer $P$ and partitions 
    $\{1 \dots m\} = I_1 \sqcup ... \sqcup I_P$,
    $\{1 \dots n\} = J_1 \sqcup ... \sqcup J_P$ such that
    \begin{itemize}
        \item For all $p \leq P$,
        $\sum_{i \in I_p} \mu_i = \sum_{j \in J_p} \nu_j$ and
        $(\restr{\mu}{I_p}, \restr{\nu}{J_p}, \EEE_p)$ is exactly scalable,%
        \footnote{Here $\restr{\mu}{I_p}, \restr{\nu}{J_p}$ denote the restrictions of the vectors $\mu, \nu$ to $I_p, J_p$ respectively. 
        The notations $X_\sharp$ and $Y_\sharp$ and \autoref{def:prelim:scalable_triplet} are extended to the case of unnormalized positive vectors in the natural way.}
        where
        \begin{equation}
            \EEE_p = \EEE \cap (I_p \times J_p).
        \end{equation}
        We call the subgraphs $(I_p \sqcup J_p, \EEE_p)$ of $(\{1 \dots m\} \sqcup \{1 \dots n\}, \EEE)$ the \emph{Dulmage-Mendelsohn (DM) components} of $(\mu, \nu, \EEE)$.
        \item Denoting by ``$\to$'' the relation on $\{1 \dots P\}$ defined by
        \begin{equation}
            p \to q
            ~~~\iff~~~
            p \neq q
            ~~\text{and}~~
            \EEE \cap (I_p \times J_q) \neq \varnothing,
        \end{equation}
        the directed graph $(\{1 \dots P\}, \{(p,q);~ p \to q\})$ is a directed acyclic graph (DAG).
        We call it the \emph{DM interaction DAG} of $(\mu, \nu, \EEE)$.
    \end{itemize}
\end{proposition}

\newcommand{\diagramscale}{0.8} % Common scale factor for all three TikZ pictures.
\begin{figure}[t]
    \centering
    \begin{subfigure}[c]{0.46\textwidth}
        \centering
        \scalebox{\diagramscale}{%
            \input{tikz_fig1a}%
        }
        \caption{An example of a bipartite graph 
        % $([m] \sqcup [n], \EEE)$ 
        with vertices arranged and grouped by DM component}
        \label{fig:mainres:DM_E_a}
    \end{subfigure}
    \hfill
    \begin{subfigure}[c]{0.31\textwidth}
        \centering
        \scalebox{\diagramscale}{
            \input{tikz_fig1b}%
        }
        \caption{A stylized version of \autoref{fig:mainres:DM_E_a}}
        \label{fig:mainres:DM_E_b}
    \end{subfigure}
    \hfill
    \begin{subfigure}[c]{0.17\textwidth}
        \centering
        \scalebox{\diagramscale}{\input{tikz_fig1c}}
        \caption{The DM interaction DAG}
        \label{fig:mainres:DM_E_c}
    \end{subfigure}
    \caption{An illustration of the DM decomposition \autoref{prop:mainres:DM_E}.
    In the case represented here, $p=3$ and the DM interaction DAG is the maximal one.
    In (a) and (b), instead of representing specific edges, for readability we used transparent bands to indicate the $I_p, J_q$ for which $\EEE \cap (I_p \times J_q) \neq \varnothing$; so for example 
    $\EEE$ does not contain any edge with one endpoint in $I_2$ and the other in $J_1$.
    }
    \label{fig:mainres:DM_E}
\end{figure}
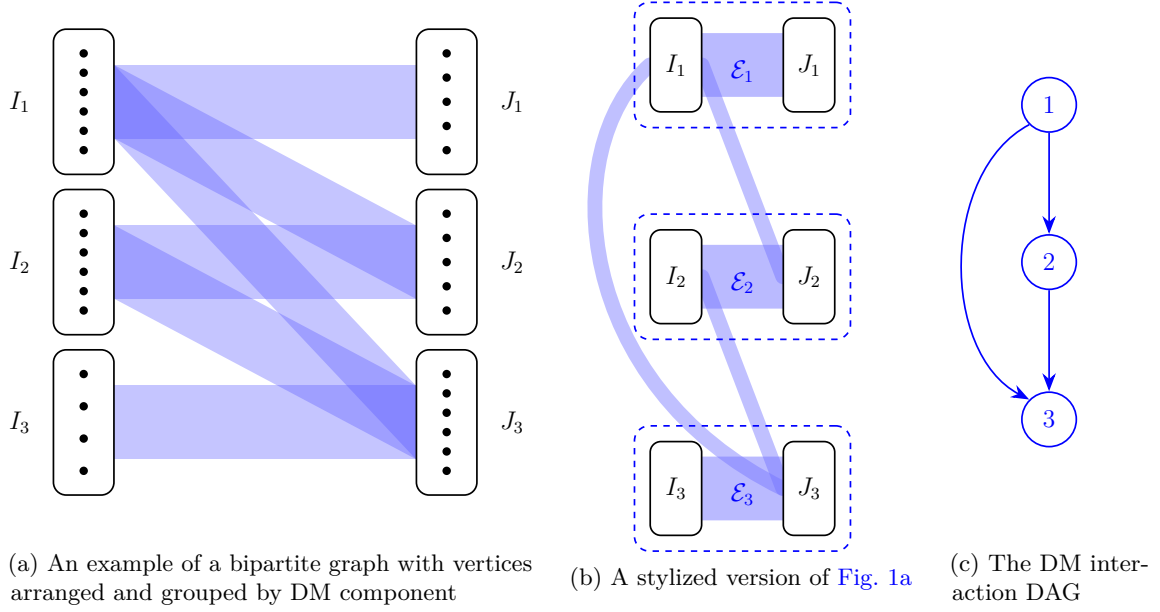

The DM decomposition also allows to capture exactly the structure of the optimal support set $\SSS = \left\{ (i,j);~ \pi^*_{ij}>0 \right\} \subset \EEE$.
Namely, $\SSS$ consists precisely of the in-optimal-support edges of $\EEE$, as the next proposition shows.
We note that both \autoref{prop:mainres:DM_S} and \autoref{prop:mainres:DM_pi*} below appeared previously in \cite[Theorem~3.4]{hayashi2024finding}, but we provide complete proofs for convenience.

\begin{proposition} \label{prop:mainres:DM_S} 
    In the setting of \autoref{prop:mainres:DM_E}, we have
    \begin{equation}
        \SSS(\mu, \nu, \EEE) = \bigsqcup_{p \leq P} \EEE_p.
    \end{equation}
\end{proposition}

\begin{proof}
    Denote for concision $\SSS = \SSS(\mu, \nu, \EEE)$.
    We proceed in two steps. 
    
    \underline{Step 1:} Let us show that the in-component edges of $\SSS$ are precisely all of the in-component edges of $\EEE$, that is,
    \begin{equation}
        \forall p \leq P,~~
        \SSS \cap (I_p \times J_p) = \EEE_p.
    \end{equation}
    The inclusion $\SSS \cap (I_p \times J_p) \subset \EEE_p$ holds by definition since $\SSS \subset \EEE$.
    For the reverse inclusion: for each $p \leq P$, $(\restr{\mu}{I_p}, \restr{\nu}{J_p}, \EEE_p)$ is exactly scalable, so by \autoref{lm:prelim:piopt_SSS} there exists $\pi^p \in (\RR_+)^{\EEE_p}$ such that $X_\sharp \pi^p = \restr{\mu}{I_p}$, $Y_\sharp \pi^p = \restr{\nu}{J_p}$ and $\forall (i,j) \in \EEE_p, \pi^p_{ij}>0$.
    Set $\pi^*_{ij} = \pi^p_{ij}$ if $(i,j) \in I_p \times J_p$ and $0$ otherwise; then $\pi^* \in \Delta_\EEE$, $X_\sharp \pi^* = \mu, Y_\sharp \pi^* = \nu$ and
    $\forall p, \forall (i,j) \in \EEE_p, \pi^*_{ij}>0$ and hence $(i,j) \in \SSS \cap (I_p \times J_p)$.

    \underline{Step 2:} It remains to show that $\SSS$ does not contain any cross-component edges, that is,
    ${\forall p \neq q,~
    \SSS \cap (I_p \times J_q) = \varnothing}$.
    Equivalently, we want to show that for any $\pi \in \Delta_\EEE$ such that $X_\sharp \pi = \mu$ and $Y_\sharp \pi = \nu$,
    \begin{equation}
        \forall p \neq q,~~
        \forall (i,j) \in \EEE \cap (I_p \times J_q),~~
        \pi_{ij} = 0.
    \end{equation}
    Fix henceforth such a $\pi$.
    Note that we may equivalently restrict to $p, q$ such that $p \to q$ in the DM interaction DAG, since $\EEE \cap (I_p \times J_q) = \varnothing$ otherwise.

    Let us show by induction that for any $q$,
    \begin{equation}
        \sum_{i \in I_q, j \in J_q} \pi_{ij} = \sum_{i \in I_q} \mu_i = \sum_{j \in J_q} \nu_j
        \qquad\text{and}\qquad
        \forall p ~~\text{s.t.}~ p \to q,~
        \forall i \in I_p, j \in J_q,~~
        \pi_{ij} = 0.
    \end{equation}
    We proceed by induction on the depth of $q$ (the maximal length of a path ending at $q$) in the DM interaction DAG.
    \begin{itemize}
        \item If $q$ is a root of this DAG, then all edges in $\EEE$ with one endpoint in $J_q$ must have the other one in $I_q$, i.e., 
        $\EEE \cap (\{1 \dots m\} \times J_q) = \EEE \cap (I_q \times J_q)$, so
        \begin{equation}
            \sum_{i \in I_q, j \in J_q} \pi_{ij}
            = \sum_{i \leq m} \sum_{j \in J_q} \pi_{ij}
            = \sum_{j \in J_q} (Y_\sharp \pi)_j
            = \sum_{j \in J_q} \nu_j.
        \end{equation}
        The second part of the desired property is trivially satisfied since there does not exist any $p$ such that $p \to q$.
        \item If $q$ is not a root of the DAG, we can assume by induction that all of its parents $p$ satisfy 
        $\sum_{i \in I_p, j \in J_p} \pi_{ij} = \sum_{i \in I_p} \mu_i$.
        Suppose by contradiction that there exists $p \to q$ and $i_0 \in I_p$, $j_0 \in J_q$ such that $\pi_{i_0j_0}>0$, then
        \begin{equation}
            \sum_{i \in I_p} \mu_i 
            = \sum_{i \in I_p} (X_\sharp \pi)_i
            = \sum_{i \in I_p} \sum_{j \leq n} \pi_{ij}
            \geq \sum_{i \in I_p, j \in J_p} \pi_{ij} + \pi_{i_0j_0}
            > \sum_{i \in I_p} \mu_i,
        \end{equation}
        a contradiction. This proves the second part of the desired property.
        To show the first part, it now suffices to note that
        \begin{equation}
            \sum_{j \in J_q} \nu_j
            = \sum_{j \in J_q} (Y_\sharp \pi)_j
            = \sum_{i \leq m} \sum_{j \in J_q} \pi_{ij}
            = \sum_{i \in I_q, j \in J_q} \pi_{ij}
            + \sum_{p:\, p \to q}~ 
            \underbrace{\sum_{i \in I_p, j \in J_q} \pi_{ij}}_{0}.
        \end{equation}
    \end{itemize}
    This concludes the proof by induction, and so the proof of Step~2 since $\pi$ was arbitrary.
\end{proof}

\pagebreak

Since $\SSS = \left\{ (i,j);~ \pi^*_{ij}>0 \right\}$ has no cross-component edges in the DM decomposition, then $\pi^*$ decomposes into $P$ ``diagonal'' blocks, each of which can be characterized as the optimal solution of an exactly scalable EOT problem. Our next proposition formalizes this.

\begin{proposition} \label{prop:mainres:DM_pi*}
    Denote by $(I_p \sqcup J_p, \EEE_p)$ the DM decomposition of $(\mu, \nu, \EEE)$ and by $(\{1 \dots P\}, \allowbreak \{(p,q);~ p \to q\})$ its DM interaction DAG.
    The optimal solution $\pi^*$ of \eqref{eq:intro:EOT} decomposes as
    \begin{equation}
        \forall p, q,~~~
        \forall i \in I_p, j \in J_q,~~~
        \pi^*_{ij} = \pi^{*p}_{ij} ~~\text{if}~ p = q
        ~~\text{and}~~ 0 ~\text{otherwise},
    \end{equation}
    where for all $p \leq P$,
    \begin{equation}
        \pi^{*p} = \argmin_{\pi \in (\RR_+)^{\EEE_p}} 
        \sum_{(i,j) \in \EEE_p} C_{ij} \pi_{ij} 
        + \tau \Hdiv{\pi}{\restr{\mu}{I_p} \otimes \restr{\nu}{J_p}}
        ~~~~\text{subject to}~~~~
        X_\sharp \pi = \restr{\mu}{I_p}
        ~~\text{and}~~
        Y_\sharp \pi = \restr{\nu}{J_p}.
    \end{equation}
    Note that $\big( \sum_{i \in I_p} \mu_i \big)^{-1} \pi^{*p} \in \Delta_{\EEE_p}$.
    Moreover, there exist $(f^{*p}, g^{*p}) \in \RR^{I_p} \times \RR^{J_p}$ such that
    \begin{equation}
        \forall p,~
        \forall (i,j) \in \EEE_p,~~
        \pi^{*p}_{ij} = e^{[-C_{ij} + f^{*p}_i + g^{*p}_j]/\tau} \mu_i \nu_j.
    \end{equation}
\end{proposition}

\begin{proof}
    It follows directly from \autoref{prop:mainres:DM_S} that $\pi^*_{ij} = 0$ for all $i \in I_p, j \in J_q$ with $p \neq q$.
    For the characterization of the ``diagonal'' blocks, first note that by definition, $\pi^*$ is also optimal for the problem \eqref{eq:intro:EOT} restricted to $\Delta_\SSS$.
    Now since $\SSS = \bigsqcup_p \EEE_p$, this restricted problem decomposes into $P$ independent optimization problems, which are the ones displayed in the proposition statement.
    Finally, the existence of finite optimal dual variables $(f^{*p}, g^{*p})$ for the restricted problems follows from the fact that $(\restr{\mu}{I_p}, \restr{\nu}{J_p}, \EEE_p)$ is exactly scalable for each $p$.
\end{proof}

\subsection{Local convergence analysis} \label{subsec:mainres:loc}

In this section, we prove the following local convergence result.

\begin{theorem} \label{thm:mainres:loc}
    Let $\pi^*_{\min} = \min_{(i,j) \in \SSS} \pi^*_{ij}$.
    Suppose there exists $k_0 \geq 1$ such that
    \begin{equation}
        \forall k \geq k_0,~ 
        \forall (i,j) \in \SSS,~
        \pi^k_{ij} \geq \pi^*_{\min}/2.
    \end{equation}
    Denote by $(I_p \sqcup J_p, \EEE_p)$ the DM decomposition of $(\mu, \nu, \EEE)$ and by $(\{1 \dots P\}, \{(p,q);~ p \to q\})$ its DM interaction DAG.
    Let $\mathrm{diam}(\SSS)$ be the maximal length of a shortest path between two vertices in $(\{1 \dots m\} \sqcup \{1 \dots n\}, \SSS)$,
    and let $\ell$ be the maximal length of a path in the DM interaction DAG.
    Then
    \begin{equation}
        \forall k \geq k_0 + 1,~~
        V_k \leq \frac{2 \tau A^2}{k-k_0}
        \qquad\text{and}\qquad
        \forall k \geq 2k_0 + 2,~~
        E_k \leq \frac{4 \sqrt{2}\, A}{\sqrt{k (k - 2k_0)}}
    \end{equation}
    where $A = \frac{\ell}{2} + (\ell+1) \mathrm{diam}(\SSS) \log(2 / \pi^*_{\min})$.
\end{theorem}

\begin{remark}
    Note that $A$ depends on $\tau$ through $\pi^*_{\min}$. More precisely, 
    % using the characterization of \autoref{prop:mainres:DM_pi*}, one can check that
    we show in \autoref{lm:mainres:estim_pi*min} below that
    $\log(1/\pi^*_{\min})$ scales as $O(\frac1\tau)$, when $\tau$ is small.
    So the upper bounds on $V_k$ and $E_k$ shown in this theorem both scale as $O(\frac{1}{\tau k})$.
\end{remark}

\begin{remark}
    Our notation $\mathrm{diam}(G)$ corresponds to the diameter of a bipartite graph $G$ in the case where it is connected.
    If it is not connected, $\mathrm{diam}(G)$ is the maximum diameter of its connected components.
    In particular, we have
    $\mathrm{diam}(\SSS) = \max_p \mathrm{diam}(\EEE_p)$,
    but note that the $\EEE_p$ themselves may not be connected in general.
    
    Let us also clarify that the notion of ``diameter'' used to define $\mathrm{diam}(\cdot)$ is different than the one for $\ell$: the former considers the shortest paths between pairs of vertices, whereas the latter is a supremum over all paths in a directed graph.
\end{remark}

\vspace{0.25em}
The proof of \autoref{thm:mainres:loc} is based on using $V_k = \tau \Hdiv{\pi^*}{\pi^k}$ as a Lyapunov potential: we will show that $V_{k+1} \leq V_k - a V_k^2$ for all $k \geq k_0$ for some constant $a>0$, and the bound $V_k \leq \frac{1/a}{k-k_0}$ will follow.
To show this, we treat separately the two terms in the decomposition
$V_k/\tau = \sum_{(i,j) \in \SSS} 
D_h(\pi^*_{ij} \| \pi^k_{ij})
+ \sum_{(i,j) \in \EEE\setminus\SSS} \pi^k_{ij}$ from \autoref{lm:prelim:decomp_Vk}.
We start by the first term, i.e., the in-optimal-support edges.

\begin{lemma} \label{lm:mainres:bound_insupport}
    Let $f \in \RR^m, g \in \RR^n$ such that $Z(f,g) = 1$ and $\min_{(i,j) \in \SSS} \pi[f,g]_{ij} \geq \pi^*_{\min}/2$.
    Then $\pi = \pi[f,g]$ satisfies
    \begin{equation}
        \sum_{(i,j) \in \SSS} D_h(\pi^*_{ij} \| \pi_{ij})
        \leq \mathrm{diam}(\SSS) \log(2 / \pi^*_{\min}) \bigg( 
            \norm{X_\sharp \pi - \mu}_1 
            + \norm{Y_\sharp \pi - \nu}_1 
            + 2 \sum_{(i,j) \in \EEE \setminus \SSS} \pi_{ij}
        \bigg).
    \end{equation}
\end{lemma}

\begin{proof}
    Let $\pi^{*p} \in (\RR_+)^{\EEE_p}, f^{*p} \in \RR^{I_p}, g^{*p} \in \RR^{J_p}$ be as defined in \autoref{prop:mainres:DM_pi*}.
    Recall that $\SSS = \bigsqcup_p \EEE_p$.
    For all $p \leq P$ and $(i,j) \in \EEE_p$,
    $\pi^*_{ij} = \pi^{*p}_{ij} = e^{[-C_{ij}+f^{*p}_i+g^{*p}_j]/\tau} \mu_i \nu_j$,
    so by definition of
    $\pi_{ij} = e^{[-C_{ij}+f_i+g_j]/\tau} \mu_i \nu_j$
    and of 
    $D_h(a \| b) = a \log \frac{a}{b} - a + b$, 
    \begin{align}
        \forall (i,j) \in \EEE_p,~~
        D_h(\pi^*_{ij} \| \pi_{ij})
        \leq D_h(\pi^*_{ij} \| \pi_{ij})
        + D_h(\pi_{ij} \| \pi^*_{ij})
        &= (\pi^*_{ij} - \pi_{ij}) \log(\pi^*_{ij} / \pi_{ij}) \\
        &= \frac1\tau (\pi^*_{ij} - \pi_{ij})
        \left( f^{*p}_i - f_i + g^{*p}_j - g_j \right).
    \end{align}
    Thus
    \begin{align}
        \tau \sum_{(i,j) \in \EEE_p}
        D_h(\pi^*_{ij} \| \pi_{ij})
        &\leq \sum_{(i,j) \in \EEE_p}
        (\pi^*_{ij} - \pi_{ij})
        \left( f^{*p}_i - f_i + g^{*p}_j - g_j \right) \\
        &=~ \sum_{i \in I_p} \sum_{j \leq n}
        (\pi^*_{ij} - \pi_{ij})
        (f^{*p}_i - f_i)
        - \sum_{i \in I_p} \sum_{j \not\in J_p}
        (0-\pi_{ij}) (f^{*p}_i - f_i) \\
        &~~~ + \sum_{i \leq m} \sum_{j \in J_p}
        (\pi^*_{ij} - \pi_{ij}) (g^{*p}_j - g_j)
        - \sum_{i \not\in I_p} \sum_{j \in J_p}
        (0-\pi_{ij}) (g^{*p}_j - g_j) \\
        &=~ \sum_{i \in I_p} \left[ (X_\sharp \pi^*)_i - (X_\sharp \pi)_i \right] (f^{*p}_i - f_i)
        + \sum_{q:\, p \to q}\, \sum_{i \in I_p, j \in J_q} \pi_{ij} (f^{*p}_i - f_i) \\
        &~~~ + \sum_{j \in J_p} \left[ (Y_\sharp \pi^*)_j - (Y_\sharp \pi)_j \right] (g^{*p}_j - g_j)
        + \sum_{o:\, o \to p}\, \sum_{i \in I_o, j \in J_p} \pi_{ij} (g^{*p}_j - g_j).
    \end{align}
    Let $\Lambda_p = \max_{i \in I_p} \abs{f^{*p}_i - f_i} \vee \max_{j \in J_p} \abs{g^{*p}_j - g_j}$ and $\Lambda = \max_p \Lambda_p$.
    Recalling that $X_\sharp \pi^* = \mu$ and $Y_\sharp \pi^* = \nu$, we have
    \begin{align*}
        \tau \!\! \sum_{(i,j) \in \SSS} \!D_h(\pi^*_{ij} \| \pi_{ij})
        = \tau \sum_p \! \sum_{(i,j) \in \EEE_p} \!\! D_h(\pi^*_{ij} \| \pi_{ij})
        &\leq \Lambda \Big(
            \norm{\mu - X_\sharp \pi}_1
            + \norm{\nu - Y_\sharp \pi}_1
            + 2\, \sum_{p \to q}\, \sum_{i \in I_p, j \in J_q} \!\! \pi_{ij} 
        \Big) \\
        &= \Lambda \Big(
            \norm{\mu - X_\sharp \pi}_1
            + \norm{\nu - Y_\sharp \pi}_1
            + 2 \sum_{(i,j) \in \EEE \setminus \SSS} \pi_{ij}
        \Big).
    \end{align*}

    It only remains to bound $\Lambda_p = \max_{i \in I_p} \abs{f^{*p}_i - f_i} \vee \max_{j \in J_p} \abs{g^{*p}_j - g_j}$ for all $p$.
    First note that
    \begin{align}
        \forall (i,j) \in \EEE_p,~~~
        f^{*p}_i - f_i + g^{*p}_j - g_j
        &= \tau \log(\pi^*_{ij} / \pi_{ij})
        \leq \tau \log(1 / \min_\SSS \pi)
        \leq \tau \log(2 / \pi^*_{\min}) \\
        \text{and} \quad
        f^{*p}_i - f_i + g^{*p}_j - g_j
        &\geq \tau \log(\pi^*_{\min}) \\
        \text{so} \quad
        \abs{f^{*p}_i - f_i + g^{*p}_j - g_j} 
        &\leq \tau \log(2 / \pi^*_{\min}).
    \end{align}
    Suppose for now that $(I_p \sqcup J_p, \EEE_p)$ is connected.
    Up to replacing $(f^{*p}, g^{*p})$ by $(f^{*p} + c \bmone_m, g^{*p} - c \bmone_n)$ for some $c \in \RR$, which does not affect $\pi^{*p} = e^{[-C_{ij} + f^{*p}_i + g^{*p}_j]/\tau} \mu_i \nu_j$,
    we may assume without loss of generality that $f^{*p}_{i_0} - f_{i_0} = 0$ for some $i_0 \in I_p$.
    Then for any neighbor $j$ of $i_0$ in $\EEE_p$, $\abs{g^{*p}_j - g_j} \leq \tau \log (2/\pi^*_{\min})$.
    Further, for any neighbor $i$ of a neighbor of $i_0$ in $\EEE_p$, $\abs{f^{*p}_i - f_i} \leq 2 \tau \log (2/\pi^*_{\min})$.
    Continuing, by induction we obtain that
    \begin{equation}
        \Lambda_p \leq \mathrm{diam}(\SSS) \cdot \tau \log (2/\pi^*_{\min}),
    \end{equation}
    since $\mathrm{diam}(\SSS)$ is an upper bound on the maximal length of a shortest path between $i_0$ and another vertex in $(I_p \sqcup J_p, \EEE_p)$.
    Finally, if $(I_p \sqcup J_p, \EEE_p)$ is not connected,
    one can apply the same reasoning to each of its connected components, choosing the gauge of one special vertex $i_0$ for each component.
\end{proof}

Next, we show a bound on the second term in the decomposition of $V_k/\tau$ from \autoref{lm:prelim:decomp_Vk}: the mass assigned by $\pi^k$ to the off-optimal-support edges,
$\sum_{(i,j) \in \EEE\setminus\SSS} \pi^k_{ij}$.

\begin{lemma} \label{lm:mainres:bound_offsupport}
    For any $\pi \in \Delta_\EEE$, we have
    $\sum_{(i,j) \in \EEE\setminus\SSS} \pi_{ij}
    \leq \frac{\ell}{2} \left( \norm{X_\sharp \pi - \mu}_1 + \norm{Y_\sharp \pi - \nu}_1 \right)$.
\end{lemma}

\begin{proof}
    Since $\SSS = \bigsqcup_p \EEE_p$, then $\EEE \setminus \SSS$ consists precisely of the cross-DM-component edges: 
    $\EEE \setminus \SSS = \bigsqcup_{p \to q} \EEE \cap (I_p \times J_q)$.
    For concision, for any $p, q \leq P$, 
    \begin{itemize}
        \item Let $\Pi_{pq} = \sum_{i \in I_p, j \in J_q} \pi_{ij}$,
        so that
        $\sum_{(i,j) \in \EEE\setminus\SSS} \pi_{ij} 
        = \sum_{p \to q} \Pi_{pq}$.
        \item Also set 
        $\Pi_{p\bullet}
        = \sum_{i \in I_p} \sum_{j \leq n} \pi_{ij}
        = \sum_{q=1}^P \Pi_{pq}$
        and likewise for $\Pi_{\bullet q}$.
        \item Correspondingly, set
        $M_p = \sum_{i \in I_p} \mu_i = \sum_{j \in J_p} \nu_j$.
        \item Furthermore, let $\ell_p$ denote the maximal length of a path ending at $p$ in the DM interaction DAG.
        Note that for any edge $p \to q$ in the DAG, we have $\ell_q - \ell_p \geq 1$.
        Let $\ell = \max_p \ell_p$, consistent with the notation in the statement of \autoref{thm:mainres:loc}.
    \end{itemize}
    Then
    \begin{align}
        \sum_{(i,j) \in \EEE\setminus\SSS} \pi_{ij}
        = \sum_{p \to q} \Pi_{pq}
        &\leq \sum_{p \to q} (\ell_q-\ell_p) \Pi_{pq} \\
        &= \sum_{q=1}^P \ell_q \left( \sum_{p:\, p \to q} \Pi_{pq} - \sum_{r:\, q \to r} \Pi_{qr} \right) \\
        &= \sum_{q=1}^P \ell_q~ \Big( \Pi_{\bullet q} - \Pi_{qq} - \Pi_{q \bullet} + \Pi_{qq} \Big)
        = \sum_{q=1}^P \ell_q \left( \Pi_{\bullet q} - \Pi_{q \bullet} \right).
    \end{align}
    Since $\sum_{q=1}^P \Pi_{\bullet q} - \Pi_{q \bullet} = 1 - 1 = 0$, we also have
    \begin{equation}
        \sum_{(i,j) \in \EEE\setminus\SSS} \pi_{ij}
        \leq \sum_{q=1}^P \left(\ell_q - \frac{\ell}{2} \right) \left( \Pi_{\bullet q} - \Pi_{q \bullet} \right).
    \end{equation}
    So by triangle inequality at the block level, since $\max_q \abs{\ell_q - \ell/2} = \ell/2$ by definition,
    and adding and subtracting $M_q = \sum_{i \in I_q} \mu_i = \sum_{j \in J_q} \nu_j$,
    \begin{align}
        \sum_{(i,j) \in \EEE\setminus\SSS} \pi_{ij}
        &\leq \frac{\ell}{2} \sum_{q=1}^P \abs{\Pi_{\bullet q} - \Pi_{q \bullet}}
        = \frac{\ell}{2} \sum_{q=1}^P \abs{\Pi_{\bullet q} - M_q - \Pi_{q \bullet} + M_q} \\
        &\leq \frac{\ell}{2} \left(
            \sum_{q=1}^P \abs{\Pi_{\bullet q} - M_q} 
            + \sum_{p=1}^P \abs{\Pi_{p \bullet} - M_p} 
        \right).
    \end{align}
    Finally, since
    $\Pi_{p\bullet} - M_p
    = \sum_{i \in I_p} \left( \sum_{j \leq n} \pi_{ij} \right)
    - \sum_{i \in I_p} \mu_i
    = \sum_{i \in I_p} (X_\sharp \pi)_i - \mu_i$
    and likewise for $\Pi_{\bullet q} - M_q$,
    by triangle inequalities at the vertex level we obtain
    \begin{equation}
        \sum_{(i,j) \in \EEE\setminus\SSS} \! \pi_{ij}
        \leq \frac{\ell}{2}
        \left(
            \sum_q \sum_{j \in J_p} \abs{(Y_\sharp \pi)_j - \nu_j}
            + \sum_p \sum_{i \in I_p} \abs{(X_\sharp \pi)_i - \mu_i}
        \right)
        = \frac{\ell}{2}
        \left(
            \norm{Y_\sharp \pi - \nu}_1
            + \norm{X_\sharp \pi - \mu}_1
        \right),
    \end{equation}
    as announced.
\end{proof}

By combining \autoref{lm:mainres:bound_insupport} and \autoref{lm:mainres:bound_offsupport}, we have shown that for all $k$ such that $\min_{\SSS} \pi^k_{ij} \geq \pi^*_{\min}/2$,
\begin{align}
    V_k/\tau 
    &= \sum_{(i,j) \in \SSS} 
    D_h(\pi^*_{ij} \| \pi^k_{ij})
    + \sum_{(i,j) \in \EEE\setminus\SSS} \pi^k_{ij} \\
    &\leq \mathrm{diam}(\SSS) \log(2 / \pi^*_{\min}) \bigg( 
        \norm{X_\sharp \pi^k - \mu}_1 
        + \norm{Y_\sharp \pi^k - \nu}_1 
        + 2 \sum_{(i,j) \in \EEE \setminus \SSS} \pi^k_{ij}
    \bigg)
    + \sum_{(i,j) \in \EEE \setminus \SSS} \pi^k_{ij} \\
    &\leq \left(
        \mathrm{diam}(\SSS) \log(2 / \pi^*_{\min}) 
        + \frac{\ell}{2} \big[ 1 + 2\, \mathrm{diam}(\SSS) \log(2 / \pi^*_{\min}) \big]
    \right)
    \left( \norm{X_\sharp \pi^k - \mu}_1 + \norm{Y_\sharp \pi^k - \nu}_1  \right).
\end{align}
So we are now in a position to apply the same derivation as done by \cite{dvurechensky2018computational} for the exactly scalable case, and this will complete the proof of \autoref{thm:mainres:loc}.
Let us show the derivation in detail.

\begin{proof}[Proof of \autoref{thm:mainres:loc}]
    By combining \autoref{lm:mainres:bound_insupport} and \autoref{lm:mainres:bound_offsupport}, we have that for all $k \geq k_0$,
    \begin{align}
        V_k/\tau
        \leq A E_k
        \qquad\text{where}\qquad
        A &= \mathrm{diam}(\SSS) \log(2 / \pi^*_{\min}) 
        + \frac{\ell}{2} \big[ 1 + 2\, \mathrm{diam}(\SSS) \log(2 / \pi^*_{\min}) \big] \\
        &= \frac{\ell}{2} + (\ell+1) \mathrm{diam}(\SSS) \log(2 / \pi^*_{\min}).
    \end{align}
    Now by \autoref{lm:prelim:Ek_Gk} and \autoref{lm:prelim:Vk+1_Vk_Gk2},
    $E_k^2 \leq 2 G_k^2 = 2 (V_k - V_{k+1})/\tau$, so
    \begin{align}
        V_k^2
        \leq \tau^2 A^2 E_k^2
        &\leq 2 \tau A^2 (V_k-V_{k+1}) \\
        V_{k+1} &\leq V_k - \frac{1}{2\tau A^2} V_k^2.
    \end{align}
    Hence by a classical argument (see, e.g., \cite[Lemma~2.4]{wang2026almost}),
    \begin{equation}
        \forall k \geq k_0+1,~~
        V_k \leq \frac{2 \tau A^2}{k-k_0}.
    \end{equation}
    This proves the first inequality announced in the theorem statement.
    To obtain the second one, it suffices to apply \autoref{lm:prelim:Ek2_Vk/2}, yielding
    \begin{equation}
        \forall k \geq 2 k_0 + 2,~~
        E_k^2 \leq \frac{8}{\tau k} V_{\ceil{k/2}}
        \leq \frac{8}{\tau k} \cdot 
        \frac{2 \tau A^2}{k/2-k_0}
        \leq \frac{32 A^2}{k (k-2 k_0)},
    \end{equation}
    as announced.
\end{proof}

\subsection{Global convergence guarantee} \label{subsec:mainres:glob}

Due to the qualitative convergence $\pi^k \to \pi^*$ shown by \cite{sinkhorn1967concerning}, necessarily the localness condition $\min_{(i,j) \in \SSS} \pi^k_{ij} \geq \pi^*_{\min}/2$ is satisfied for all $k$ greater than some $k_0$.
So \autoref{thm:mainres:loc} already shows that $E_k, V_k \leq O(1/k)$ in the regime $k \to \infty$.
But to show a computational complexity bound, one still needs to estimate this iterate $k_0$ explicitly.
This is the object of this section.

Let us first recall the state-of-the-art global convergence bound.

\begin{theorem}[{\cite[Theorem~3.2]{wang2026almost}}] \label{thm:mainres:wan26}
    Pose $\osc_\EEE(C) = \max_{(i,j) \in \EEE} C_{ij} - \min_{ij} C_{ij}$
    and
    $\Delta = \min_{\substack{I \subset \{1 \dots m\} \\ J \subset \{1 \dots n\}}}~
    \abs{\sum_{i \in I} \mu_i - \sum_{j \in J} \nu_j}$
    subject to
    $\sum_{i \in I} \mu_i \neq \sum_{j \in J} \nu_j$.
    Let $\ell$ be as defined in \autoref{thm:mainres:loc}.
    Then $V_k = \tau \Hdiv{\pi^*}{\pi^k}$ satisfies
    \begin{equation}
        \forall k \geq e \ell^2 + 1,~~
        V_k/\tau \leq \frac{2}{k-1} 
        \left[ 
            B + \ell\, \log \frac{k-1}{\ell^2}
        \right]^2
    \end{equation}
    where
    $B = \left( \frac{\osc_\EEE(C)}{\tau} - \log (\mu_{\min} \vee \nu_{\min}) \right)
    \left( \ell + \frac{2 (\ell + 1)}{\Delta} \right)$.
\end{theorem}

In order to get a fully explicit bound, let us also show the following estimate on $\pi^*_{\min}$.

\begin{lemma} \label{lm:mainres:estim_pi*min}
    Let $\osc_\EEE(C)$ and $\Delta$ be as defined in \autoref{thm:mainres:wan26}.
    Then $\pi^*_{\min} = \min_{(i,j) \in \SSS} \pi^*_{ij}$ satisfies
    \begin{equation}
        -\log \pi^*_{\min} 
        \leq \frac{2}{\Delta} \left( \frac{\osc_{\EEE}(C)}{\tau} - \log(\mu_{\min} \vee \nu_{\min}) \right)
        - \log(\mu_{\min} \vee \nu_{\min})
        ~\eqqcolon D.
    \end{equation}
\end{lemma}

\begin{proof}
    By \autoref{prop:mainres:DM_S}, $\SSS = \bigsqcup_p \EEE_p$ and by \autoref{prop:mainres:DM_pi*}, we have
    % $\pi^*_{ij} = 0$ for all $i \in I_p, j \in J_q$ with $p \neq q$ and
    \begin{equation}
        \forall p,~~ \forall (i,j) \in \EEE_p,~~
        \pi^*_{ij} = \pi^{*p}_{ij} = e^{[-C_{ij} + f^{*p}_i + g^{*p}_j]/\tau} \mu_i \nu_j.
    \end{equation}
    Now by \cite[Proposition~2.7, Lemma~A.2]{wang2026almost},
    denoting $K_p = \max_{(i,j) \in \EEE_p} C_{ij} - \tau \log(\mu_{\min} \vee \nu_{\min})$,
    % (adapted from \cite[Theorem~5.1]{kalantari2008complexity}), 
    the $(f^{*p}, g^{*p})$ can be chosen such that for each $p$,
    \begin{equation}
        \max_{i \in I_p} \abs{f^{*p}_i - \frac{K_p}{2}},~~
        \max_{j \in J_p} \abs{g^{*p}_j - \frac{K_p}{2}}
        \leq \frac{1}{\Delta} \big( \osc_{\EEE_p}(C) - \tau \log(\mu_{\min} \vee \nu_{\min}) \big).
    \end{equation}
    Hence,
    \begin{multline}
        \forall (i,j) \in \EEE_p,~~
        -\tau \log \pi^{*p}_{ij}
        = C_{ij} - f^{*p}_i - g^{*p}_j - \tau \log(\mu_i \nu_j)
        \leq K_p - f^{*p}_i - g^{*p}_j
        - \tau \log(\mu_{\min} \vee \nu_{\min}) \\
        \leq \frac{2}{\Delta} \big( \osc_{\EEE_p}(C) - \tau \log(\mu_{\min} \vee \nu_{\min}) \big)
        - \tau \log(\mu_{\min} \vee \nu_{\min})
    \end{multline}
    and thus
    \begin{equation}
        -\tau \log \pi^*_{\min}
        = \max_p \max_{(i,j) \in \EEE_p} \left( -\tau \log \pi^{*p}_{ij} \right)
        \leq \frac{2}{\Delta} \big( \osc_{\EEE}(C) - \tau \log(\mu_{\min} \vee \nu_{\min}) \big)
        - \tau \log(\mu_{\min} \vee \nu_{\min})
    \end{equation}
    since $\max_p \osc_{\EEE_p}(C) \leq \osc_{\EEE}(C)$ by definition.
\end{proof}

By combining the global convergence bound of \autoref{thm:mainres:wan26} with the local convergence analysis of \autoref{subsec:mainres:loc}, we obtain the following improved global bound. The proof consists in explicit computations and is delayed to \autoref{sec:apx_delayedpfs}.

\begin{theorem} \label{thm:mainres:glob}
    Let $A$ and $B$ be the constants defined in \autoref{thm:mainres:loc} and \autoref{thm:mainres:wan26}:
    $A = \frac{\ell}{2} + (\ell+1) \mathrm{diam}(\SSS) \log(2 / \pi^*_{\min})$,
    $B = \left( \frac{\osc_\EEE(C)}{\tau} - \log (\mu_{\min} \vee \nu_{\min}) \right)
    \left( \ell + \frac{2 (\ell + 1)}{\Delta} \right)$.
    Let $D$ denote the upper bound on $\log(1/\pi^*_{\min})$ shown in  \autoref{lm:mainres:estim_pi*min}:
    $D = \frac{2}{\Delta} \left( \frac{\osc_{\EEE}(C)}{\tau} - \log(\mu_{\min} \vee \nu_{\min}) \right)
    - \log(\mu_{\min} \vee \nu_{\min})$.
    Then we have
    \begin{equation}
        \forall k \geq 2 e \ell^2 + 6,~~
        E_k 
        \leq \sqrt{\frac{8 V_{\ceil{k/2}}}{\tau k}}
        \leq \frac{8}{k} \left(
            A + B 
            + 2 \ell \,\bigg(
                \log 64
                + 2 D + \log(B) + \log (D + \log 32)
            \bigg)
        \right).
    \end{equation}
\end{theorem}

In order to deduce the statement of \autoref{thm:intro:mainres}, it suffices to determine the dependency of $A, B$, and $D$ on the problem parameters.
Let us use $\Theta(\cdot)$ to hide constants dependent only on $\mu, \nu$, and $\EEE$, then
\begin{align}
    A &= \frac{\ell}{2} + (\ell+1) \mathrm{diam}(\SSS) ~\underbrace{\log(2 / \pi^*_{\min})}_{\leq \log 2 + D}~
    \leq \Theta(1 + D), \\
    B &= \left( \frac{\osc_\EEE(C)}{\tau} - \log (\mu_{\min} \vee \nu_{\min}) \right)
    \left( \ell + \frac{2 (\ell + 1)}{\Delta} \right)
    = \Theta\Big( 1 + \frac{\osc_\EEE(C)}{\tau} \Big), \\
    D &= \frac{2}{\Delta} \left( \frac{\osc_{\EEE}(C)}{\tau} - \log(\mu_{\min} \vee \nu_{\min}) \right)
    - \log(\mu_{\min} \vee \nu_{\min})
    = \Theta\Big( 1 + \frac{\osc_\EEE(C)}{\tau} \Big).
\end{align}
So the bound on $E_k$ proved in \autoref{thm:mainres:glob} reads
\begin{equation}
    k \cdot E_k \leq \Theta\left( A + B + D + \log B + \log D \right)
    \leq \Theta\left( A + B + D \right)
    = \Theta\Big( 1 + \frac{\osc_\EEE(C)}{\tau} \Big),
\end{equation}
and hence the inequality announced in \autoref{thm:intro:mainres}.

% \section{Extension to the non-scalable case} \label{sec:nonscal}
% \TODO{}
% 
% % When $\SSS = \EEE$, we say we are in the exactly scalable case; when $\SSS \neq \EEE$, we say we are in the strictly asymptotically scalable case.

\section{Conclusion} \label{sec:ccl}

In this work, we have shown that the Sinkhorn algorithm converges at the sharp rate of $O(1/k)$ in $\ell_1$-norm marginal error and in dual suboptimality.
This matches the lower bound in the strictly asymptotically scalable case \cite{qu2025sinkhorn}.
The proof is based on a local analysis, bootstrapped by the previous result of \cite{wang2026almost} that showed a global rate of $O(\log k/k)$.

This paper settles the question 
% open since \cite{sinkhorn1967concerning},
of the Sinkhorn algorithm's convergence rate for matrix balancing, or equivalently, for the classical setting of EOT as described in \autoref{sec:intro} (with potentially infinite transport costs).
The analogous questions for the multi-marginal or unbalanced variants of Sinkhorn remain open, and it would be interesting to see whether a similar analysis would go through.

% Additionally, as pointed out in \cite[Remark~3.2]{wang2026almost}, a similar logarithmic gap between upper bound ($O(\log k/k)$) and lower bound ($\Omega(1/k)$) exists in ......
% possibility of using similar ideas to shave off the $\log$ in the convergence bound for \cite{chizat2022convergence}
% ---> actually I want to think about this myself first

\section*{Acknowledgments}

A LLM (ChatGPT-5.5 Thinking) was used throughout this project, except at the writing stage (save for tikz code generation).
All of the key proof ideas were provided by it in some form.
The main contribution of the author was to reorganize and present the arguments in a comprehensible way.
The author assumes responsibility for all content.

% \newpage
% \nocite{*}
\printbibliography
\addcontentsline{toc}{section}{\refname} % Add References to (ToC and) bookmarks

%%%%%%%%%%%%%%%%%%%%%%%%%%%%%%%%%%%%%%%%%%%%%%%%%%%%%%
\ifextended%
\newpage
\appendix
\phantomsection
% \addcontentsline{toc}{chapter}{APPENDIX}
\addcontentsline{toc}{section}{APPENDIX}

%\subfile{appendix/A_}
%\subfile{appendix/B_}

% \section{Delayed proofs} \label{sec:apx_delayedpfs}
% \subsection{Proof of \autoref{thm:mainres:glob}}

\section{Proof of \autoref{thm:mainres:glob}} \label{sec:apx_delayedpfs}

In this section we provide the proof of \autoref{thm:mainres:glob}, restated below.

\begin{theorem*}[\autoref{thm:mainres:glob}, restated]
    Let $A$ and $B$ be the constants defined in \autoref{thm:mainres:loc} and \autoref{thm:mainres:wan26}:
    $A = \frac{\ell}{2} + (\ell+1) \mathrm{diam}(\SSS) \log(2 / \pi^*_{\min})$,
    $B = \left( \frac{\osc_\EEE(C)}{\tau} - \log (\mu_{\min} \vee \nu_{\min}) \right)
    \left( \ell + \frac{2 (\ell + 1)}{\Delta} \right)$.
    Let $D$ denote the upper bound on $\log(1/\pi^*_{\min})$ shown in  \autoref{lm:mainres:estim_pi*min}:
    $D = \frac{2}{\Delta} \left( \frac{\osc_{\EEE}(C)}{\tau} - \log(\mu_{\min} \vee \nu_{\min}) \right)
    - \log(\mu_{\min} \vee \nu_{\min})$.
    Then we have
    \begin{equation}
        \forall k \geq 2 e \ell^2 + 6,~~
        E_k 
        \leq \sqrt{\frac{8 V_{\ceil{k/2}}}{\tau k}}
        \leq \frac{8}{k} \left(
            A + B 
            + 2 \ell \,\bigg(
                \log 64
                + 2 D + \log(B) + \log (D + \log 32)
            \bigg)
        \right).
    \end{equation}
\end{theorem*}

\begin{proof}
    By Pinsker's inequality and \autoref{thm:mainres:wan26},
    \begin{equation}
        \abs{\min_{(i,j) \in \SSS} \pi^k_{ij}
        - \pi^*_{\min}}
        \leq \norm{\pi^* - \pi^k}_1
        \leq \sqrt{2 \Hdiv{\pi^*}{\pi^k}}
        = \sqrt{2 V_k/\tau} \\
        \leq \frac{2}{\sqrt{k-1}} \left[ B + \ell \, \log \frac{k-1}{\ell^2} \right].
    \end{equation}
    Pose $t = \sqrt{k-1}/\ell$, then the right-hand side rewrites
    $\frac{2B}{\ell\, t} + \frac{4 \log t}{t}$.
    Now,
    \begin{equation}
        t \geq \frac{8 B e^D}{\ell}
        ~\implies~
        \frac{2B}{\ell\, t} \leq e^{-D}/4 \leq \pi^*_{\min}/4
    \end{equation}
    and by \autoref{lm:apx_delayedpfs:bound_logt/t} below,
    \begin{equation}
        t \geq 32 e^D \log(32 e^D)
        ~\implies~
        \frac{4 \log t}{t} \leq e^{-D}/4 \leq \pi^*_{\min}/4.
    \end{equation}
    Thus, 
    $\min_{(i,j) \in \SSS} \pi^k_{ij} \geq \frac{\pi^*_{\min}}{2}$ for all
    \begin{align}
        t = \sqrt{k-1}/\ell &\geq \frac{8 B e^D}{\ell} \vee 32 e^D \log(32 e^D), \\
        \text{i.e.,}\qquad
        k &\geq k_0 = 1 + \ell^2 \left( \frac{8 B e^D}{\ell} \vee 32 e^D \log(32 e^D) \right)^2.
    \end{align}

    Consequently, for all $k \geq 2 k_0$ we have by \autoref{thm:mainres:loc}
    \begin{equation}
        V_k/\tau \leq \frac{2 A^2}{k-k_0}
        \leq \frac{4 A^2}{k},
    \end{equation}
    and for $e \ell^2 + 3 \leq k \leq 2 k_0$ the bound from \autoref{thm:mainres:wan26} \cite{wang2026almost} applies:
    \begin{equation}
        V_k/\tau \leq \frac{2}{k-1} 
        \left[ 
            B + \ell\, \log \frac{k-1}{\ell^2}
        \right]^2
        \leq \frac{4}{k}
        \left[
            B + \ell\, \log \frac{2 (k-2)}{\ell^2}
        \right]^2.
    \end{equation}
    This can be summarized, at the cost of some tightness in the constants, by the following bound valid for all $k \geq e \ell^2 + 3$:
    \begin{equation}
        V_k/\tau \leq \frac{4}{k} \left[ A + B + \ell\, \log \frac{2 (2 k_0 - 2)}{\ell^2} \right]^2.
    \end{equation}
    More explicitly, the $\log$ term reads
    \begin{align}
        \ell \log \frac{2 (2k_0-2)}{\ell^2}
        &= \ell \log \left[ 4 \left( \frac{8 B e^D}{\ell} \vee 32 \, e^D \log(32 e^D) \right)^2 \right] \\
        &= 2 \ell \log \left[ 2 \left( \frac{8 B e^D}{\ell} \vee 32 \, e^D \log(32 e^D) \right) \right] \\
        &\leq 2 \ell \left( 
            \log 64
            + \log(B e^D / \ell)
            + \log \left( e^D \log(32 e^D) \right)
        \right) \\
        &\leq 2 \ell \,\bigg(
            \log 64
            + 2 D + \log(B) + \log (D + \log 32)
        \bigg).
    \end{align}
    Finally, we have shown that
    \begin{equation}
        \forall k \geq e \ell^2 + 3,~~
        V_k/\tau \leq \frac{4}{k} \left[
            A + B 
            + 2 \ell \,\bigg(
                \log 64
                + 2 D + \log(B) + \log (D + \log 32)
            \bigg)
        \right]^2,
    \end{equation}
    and so by \autoref{lm:prelim:Ek2_Vk/2},
    \begin{equation}
        \forall k \geq 2 e \ell^2 + 6,~~
        E_k \leq \sqrt{\frac{8 V_{\ceil{k/2}}}{\tau k}}
        \leq \frac{8}{k} \left[ 
            A + B 
            + 2 \ell \,\bigg(
                \log 64
                + 2 D + \log(B) + \log (D + \log 32)
            \bigg) 
        \right]
    \end{equation}
    as announced.
\end{proof}

\begin{lemma} \label{lm:apx_delayedpfs:bound_logt/t}
    For any $0 < c \leq 1$, we have
    $\frac{\log t}{t} \leq \frac{c}{16}$
    for all
    $t \geq \frac{32}{c} \log \frac{32}{c}$.
\end{lemma}

\begin{proof}
    Pose $M = \frac{32}{c} \geq 32$, $S = M \log M$ and let $t \geq S$.
    Then $S \geq 32 \log 32 \geq e$, 
    and since $x \mapsto \frac{\log x}{x}$ is decreasing on $[e, +\infty)$, then
    $\frac{\log t}{t} \leq \frac{\log S}{S}$.
    Now $\frac{\log S}{S} = \frac{\log M + \log \log M}{M \log M} = \frac1M (1 + \frac{\log \log M}{\log M}) \leq \frac2M$, because one can check that $\frac{\log x}{x} \leq 1$ for all $x$.
    Thus $\frac{\log t}{t} \leq \frac{\log S}{S} \leq \frac2M = \frac{c}{16}$, as claimed.
\end{proof}

\fi

\end{document}

%% file: tikz_fig1a.tex
\begin{tikzpicture}[
    vertex set/.style={
        draw,
        fill=white,
        rounded corners=6pt,
        minimum width=1.0cm,
        minimum height=2.4cm,
        inner sep=0pt,
        thick
    },
    vertex/.style={
        circle,
        fill=black,
        inner sep=1.4pt
    },
    edge band/.style={
        fill=blue!65,
        fill opacity=0.35,
        draw=none
    },
    set label/.style={
        font=\large
    }
]

% All three pictures use the same vertical bounds.
\path[use as bounding box]
    (-1.2,-4.5) rectangle (7.3,4.5);

% Draw a vertex set.
% #1 = node name
% #2 = position
% #3 = label
% #4 = label position
% #5 = number of vertices
\newcommand{\VertexSet}[5]{%
    \node[
        vertex set,
        label={[set label,label distance=7pt]#4:#3}
    ] (#1) at #2 {};

    \foreach \k in {1,...,#5} {
        \pgfmathsetmacro{\yoffset}{
            0.8 - 1.6*(\k-1)/(#5-1)
        }
        \node[vertex]
            at ($(#1.center)+(0,\yoffset)$) {};
    }
}

% Broad band joining most of the facing sides.
% #1 = left box
% #2 = right box
\newcommand{\EdgeBand}[2]{%
    \def\bandInset{0.6cm}

    \path[edge band]
        ($(#1.north east)+(0,-\bandInset)$) --
        ($(#2.north west)+(0,-\bandInset)$) --
        ($(#2.south west)+(0,\bandInset)$) --
        ($(#1.south east)+(0,\bandInset)$) --
        cycle;
}

% Box positions
\def\rowSpacing{2.65}

\coordinate (I1pos) at (0,  \rowSpacing);
\coordinate (I2pos) at (0,  0);
\coordinate (I3pos) at (0, -\rowSpacing);

\coordinate (J1pos) at (6,  \rowSpacing);
\coordinate (J2pos) at (6,  0);
\coordinate (J3pos) at (6, -\rowSpacing);

% Invisible boxes used to determine the bands.
\begin{scope}[
    every node/.style={
        minimum width=1.0cm,
        minimum height=2.4cm,
        inner sep=0pt
    }
]
    \node (I1guide) at (I1pos) {};
    \node (I2guide) at (I2pos) {};
    \node (I3guide) at (I3pos) {};

    \node (J1guide) at (J1pos) {};
    \node (J2guide) at (J2pos) {};
    \node (J3guide) at (J3pos) {};
\end{scope}

% Bands are drawn behind the boxes.
\begin{scope}[on background layer]
    \EdgeBand{I1guide}{J1guide}
    \EdgeBand{I1guide}{J2guide}
    \EdgeBand{I1guide}{J3guide}

    \EdgeBand{I2guide}{J2guide}
    \EdgeBand{I2guide}{J3guide}

    \EdgeBand{I3guide}{J3guide}
\end{scope}

% Visible vertex sets.
\VertexSet{I1}{(I1pos)}{$I_1$}{left}{6}
\VertexSet{I2}{(I2pos)}{$I_2$}{left}{6}
\VertexSet{I3}{(I3pos)}{$I_3$}{left}{4}

\VertexSet{J1}{(J1pos)}{$J_1$}{right}{5}
\VertexSet{J2}{(J2pos)}{$J_2$}{right}{5}
\VertexSet{J3}{(J3pos)}{$J_3$}{right}{6}

\end{tikzpicture}

%% file: tikz_fig1b.tex
\begin{tikzpicture}[
    small set/.style={
        draw,
        fill=white,
        rounded corners=5pt,
        minimum width=0.85cm,
        minimum height=1.55cm,
        thick,
        font=\large
    },
    group box/.style={
        draw=blue,
        dashed,
        rounded corners=7pt,
        thick,
        inner sep=7pt
    },
    internal band/.style={
        blue!60,
        opacity=0.45,
        line width=30pt,
        line cap=round
    },
    external band/.style={
        blue!55,
        opacity=0.45,
        line width=7pt,
        line cap=round
    },
    epsilon label/.style={
        text=blue,
        font=\large
    }
]

% Same vertical bounds as the other two pictures.
\path[use as bounding box]
    (-1.5,-4.5) rectangle (3.7,4.5);

% Horizontal and vertical spacing.
\def\columnSpacing{2.2}
\def\rowSpacing{3.5}

% Draw one grouped pair.
% #1 = index
% #2 = vertical coordinate
\newcommand{\PairGroup}[2]{%
    \coordinate (I#1pos) at (0,#2);
    \coordinate (J#1pos) at (\columnSpacing,#2);

    % Invisible guides.
    \node[
        minimum width=0.85cm,
        minimum height=1.55cm
    ] (I#1guide) at (I#1pos) {};

    \node[
        minimum width=0.85cm,
        minimum height=1.55cm
    ] (J#1guide) at (J#1pos) {};

    % Intra-group band.
    \begin{scope}[on background layer]
        \draw[internal band]
            (I#1guide.east) -- (J#1guide.west);
    \end{scope}

    % The two vertex subsets.
    \node[small set] (I#1) at (I#1pos) {$I_{#1}$};
    \node[small set] (J#1) at (J#1pos) {$J_{#1}$};

    % Internal edge-family label.
    \node[epsilon label] at
        ($(I#1)!0.5!(J#1)+(0,-0.1)$)
        {$\mathcal{E}_{#1}$};

    % Dashed enclosure.
    \node[group box, fit=(I#1)(J#1)] (G#1) {};
}

% Three groups.
\PairGroup{1}{ \rowSpacing}
\PairGroup{2}{0}
\PairGroup{3}{-\rowSpacing}

% Inter-group bands.
\begin{scope}[on background layer]

    % Curved band from I_1 to J_3.
    \draw[external band]
        (I1guide.west)
        .. controls (-2.0,2.1) and (-1.7,-1.8) ..
        (J3guide.west);

    % Straight band from I_1 to J_2.
    \draw[external band]
        (I1guide.east) -- (J2guide.west);

    % Straight band from I_2 to J_3.
    \draw[external band]
        (I2guide.east) -- (J3guide.west);
        
\end{scope}

\end{tikzpicture}

%% file: tikz_fig1c.tex
\begin{tikzpicture}[
    vertex/.style={
        circle,
        draw=blue,
        text=blue,
        thick,
        minimum size=0.9cm,
        inner sep=0pt,
        font=\large
    },
    edge/.style={
        blue,
        thick,
        -{Stealth[length=3mm]}
    }
]

% Same vertical bounds as the other two pictures.
\path[use as bounding box]
    (-1.6,-4.5) rectangle (1.6,4.5);

% Vertices.
\node[vertex] (v1) at (0,  2.6) {$1$};
\node[vertex] (v2) at (0,  0.0) {$2$};
\node[vertex] (v3) at (0, -2.6) {$3$};

% Straight edges.
\draw[edge] (v1.south) -- (v2.north);
\draw[edge] (v2.south) -- (v3.north);

% Direct curved edge from 1 to 3.
\draw[edge]
    (v1.south west)
    .. controls (-1.8,1.5) and (-1.8,-1.5) ..
    (v3.north west);

\end{tikzpicture}

%% file: references.bib
@article{altschuler2017near,
  title={Near-linear time approximation algorithms for optimal transport via Sinkhorn iteration},
  author={Altschuler, Jason and Niles-Weed, Jonathan and Rigollet, Philippe},
  journal={Advances in neural information processing systems},
  volume={30},
  year={2017}
}

@inproceedings{dvurechensky2018computational,
  title={Computational optimal transport: Complexity by accelerated gradient descent is better than by Sinkhorn’s algorithm},
  author={Dvurechensky, Pavel and Gasnikov, Alexander and Kroshnin, Alexey},
  booktitle={International conference on machine learning},
  pages={1367--1376},
  year={2018},
  organization={PMLR}
}

@article{ghosal2025convergence,
  title={On the convergence rate of Sinkhorn’s algorithm},
  author={Ghosal, Promit and Nutz, Marcel},
  journal={Mathematics of Operations Research},
  year={2025},
  publisher={INFORMS}
}

@article{idel2016review,
  title={A review of matrix scaling and Sinkhorn's normal form for matrices and positive maps},
  author={Idel, Martin},
  journal={arXiv preprint arXiv:1609.06349},
  year={2016}
}

@article{leger2021gradient,
  title={A gradient descent perspective on Sinkhorn},
  author={L{\'e}ger, Flavien},
  journal={Applied Mathematics \& Optimization},
  volume={84},
  number={2},
  pages={1843--1855},
  year={2021},
  publisher={Springer}
}

@article{qu2025sinkhorn,
  title={On Sinkhorn’s algorithm and choice modeling},
  author={Qu, Zhaonan and Galichon, Alfred and Gao, Wenzhi and Ugander, Johan},
  journal={Operations Research},
  year={2025},
  publisher={INFORMS}
}

@inproceedings{allen2017much,
  title={Much faster algorithms for matrix scaling},
  author={Allen-Zhu, Zeyuan and Li, Yuanzhi and Oliveira, Rafael and Wigderson, Avi},
  booktitle={2017 IEEE 58th Annual Symposium on Foundations of Computer Science (FOCS)},
  pages={890--901},
  year={2017},
  organization={IEEE}
}

@article{sinkhorn1967concerning,
  title={Concerning nonnegative matrices and doubly stochastic matrices},
  author={Sinkhorn, Richard and Knopp, Paul},
  journal={Pacific Journal of Mathematics},
  volume={21},
  number={2},
  pages={343--348},
  year={1967},
  publisher={Mathematical Sciences Publishers}
}

@article{soules1991rate,
  title={The rate of convergence of Sinkhorn balancing},
  author={Soules, George W},
  journal={Linear algebra and its applications},
  volume={150},
  pages={3--40},
  year={1991},
  publisher={Elsevier}
}

@article{hayashi2024finding,
  title={Finding Hall blockers by matrix scaling},
  author={Hayashi, Koyo and Hirai, Hiroshi and Sakabe, Keiya},
  journal={Mathematics of Operations Research},
  volume={49},
  number={4},
  pages={2166--2179},
  year={2024},
  publisher={INFORMS}
}

@article{wang2026almost,
  title={Almost-sharp $O(k^{-1} \log k)$ convergence rate for the Sinkhorn algorithm in the asymptotically scalable case},
  author={Wang, Guillaume},
  journal={arXiv preprint arXiv:2604.26265},
  year={2026}
}
